\documentclass{article}

\usepackage{amsmath, amssymb, amsthm}
\usepackage{cite}
\usepackage{authblk}

\newtheorem{theorem}{Theorem}

\newtheorem{corl}[theorem]{Corollary}

\newtheorem*{theorem*}{Theorem}
\newtheorem*{corl*}{Corollary}

\numberwithin{theorem}{section}
\numberwithin{equation}{section}

\newcommand{\fav}{\operatorname{Fav}}
\newcommand{\kn}{\mathcal{K}_n}
\newcommand{\dist}{\operatorname{dist}}
\newcommand{\dimh}{\operatorname{dim}_{\mathcal{H}}}

\title{Geometric Bounds for Favard Length}
\author{Tyler Bongers %
	\thanks{Email: \texttt{charlesb@math.msu.edu} }
}

\affil{Department of Mathematics, Michigan State University}

\date{\today}

\begin{document}
\maketitle
\begin{abstract}
Given a set in the plane, the average length of its projections over all directions is called Favard length. This quantity measures the size of a set, and is closely related to metric and geometric properties of the set such as rectifiability, Hausdorff dimension, and analytic capacity. In this paper, we develop new geometric techniques for estimating Favard length. We will give a short geometrically motivated proof relating Hausdorff dimension to the decay rate of the Favard length of neighborhoods of a set. We will also show that the sequence of Favard lengths of the generations of a self-similar set is convex; this has direct applications to giving lower bounds on Favard length for various fractal sets.
\end{abstract}

\section{Introduction}
Given a set $E$ in the plane, its Favard length is the average
$$\operatorname{Fav}(E) = \int_0^{2\pi} |\pi_{\theta} E|\, d\theta$$
where $\pi_{\theta}$ is orthogonal projection onto a line $L_{\theta}$ through the origin at angle $\theta$ to the positive $x$-axis, and $|.|$ is the length measure within the line $L_{\theta}$. This quantity is comparable to the Buffon needle probability of the set $E$; this is the probability that a needle dropped near the set $E$ passes through it. 

The Favard length of a set carries a great deal of metric and geometric information about the set. It is deeply related to rectifiability; Bescovitch proved in \cite{Bes39} that a set with positive and finite length is purely unrectifiable if and only if it has Favard length zero. In such a case, we know that the Favard lengths of the $r$-neighborhoods of $E$ (that is, the set $E(r)$ of points of distance no more than $r$ from $E$) must tend to zero as $r$ does. The exact rate of decay is another measure of the size of a set, and is related to Minkowski dimension. It is also conjectured that Favard length is controlled by analytic capacity in many circumstances.

In this paper, we will give geometrically motivated proofs for various properties of Favard length. First, we will reprove a result of Mattila from \cite{Mat90} that connects the decay rate of the Favard lengths of the neighborhoods of a set with the Hausdorff dimension of the underlying set:
\begin{theorem*}
Fix $s \in (0, 1)$ and suppose that $E \subseteq \mathbb{R}^2$ is measurable, and $A \subseteq S^1$ is measurable with positive (arc-length) measure. Suppose there exists a sequence of scales $r_n \to 0$ such that 
$$\int_A \left|\pi_{\theta} \left(E(r_n)\right) \right| \, d\theta \le Cr_n^{s}$$
for some $C < \infty$. Then $\dimh E \le 1 - s$.
\end{theorem*}
The original proof of this theorem relies on potentials and estimates of the energy of a measure; we will prove this result here with a direct geometric argument. 

In the next section, we will show how self-similarity leads to new and useful properties of the sequence of Favard lengths. In particular, we will show that:
\begin{theorem*}
Suppose that $\{A_n\}_{n \in \mathbb{N}}$ is a sequence of sets such that $A_{n + 1} \subseteq A_{n}$ for all $n$, that each generation can be written as a union
$$A_{n + 1} = \bigcup_{i = 1}^N r_{i} A_n + \beta_{i}$$
for some fixed set of contraction ratios $r_i > 0$, and that $\sum_{i} r_{i} = 1$. Then for each $\theta$, the sequence $\{|\pi_{\theta} A_n|\}_{n \in \mathbb{N}}$ is convex.
\end{theorem*}
Note that since the sum of contraction ratios is $1$, the sequence of sets converges to an attractor which is a self-similar set of Hausdorff dimension at most $1$ (and if the similitudes satisfy the open set condition, it has Hausdorff dimension equal to $1$). See, e.g., Chapter 4 of \cite{Mat95} for more details.

Convexity gives a powerful constraint on the decay rate of Favard lengths: the decay within the first few generations controls the decay until much later stages. In particular, it is very easy to recover the result that:
\begin{corl*}
If $\kn$ is the $n$-th generation of the four-corner Cantor set, then $\fav(\kn) \gtrsim 1/n$. 
\end{corl*}

Before we begin the proofs, we first define some notation. Given a set $A$, we will denote its Lebesgue measure by $|A|$; depending on the context, this could mean the Lebesgue measure within a line, or area measure in the plane, or arc-length measure in the circle. If it is clear from context which one of these is meant, we will not specify.

\section{Dimension and Favard Length}
In this section, we will prove that sufficiently quick decay of Favard length of neighborhoods of a set controls the Hausdorff dimension of the set. For a set $E$ in some Euclidean space and $r > 0$, we denote the $r$-neighborhood of $E$ by
$$E(r) = \{x : \dist(x, E) < r\}.$$
The decay rate of the Lebesgue measure of $E(r)$ as $r \to 0$ is connected with the Minkowski dimension of the underlying set, as well as other notions of size. Our main result is a new proof of the following theorem:
\begin{theorem}
Fix $s \in (0, 1)$ and suppose that $E \subseteq \mathbb{R}^2$ is measurable, and $A \subseteq S^1$ is measurable with positive (arc-length) measure. Suppose there exists a sequence of scales $r_n \to 0$ such that 
$$\int_A \left|\pi_{\theta} \left(E(r_n)\right) \right| \, d\theta \le Cr_n^{s}$$
for some $C < \infty$. Then $\dimh E \le 1 - s$.
\end{theorem}

The contrapositive of this theorem appeared in \cite{Mat90} at the level of measure; here, it is only at the level of dimension. Mattila's argument relies on studying the energy of a measure; here, we use a direct geometric argument. The previous proof relies on being able to find a measure supported on the set that satisfies certain decay conditions, which is guaranteed for compact sets by Frostman's lemma. Our technique has the advantage of avoiding questions of the existence of such a measure, so we do not need any additional topological assumptions about the set.

\begin{proof}
We proceed in three steps. First, we need to find a particular direction where the projection $\pi_{\theta} E$ has full dimension while simultaneously having almost sufficiently quick decay of $|\pi_{\theta} E(r_n)|$. Secondly, we will use a H\"older inequality to control the sum of lengths over a natural cover on the projection side; this gives control on the Hausdorff measure of the projection. Finally, we tighten the bounds by adjusting exactly how quickly $|\pi_{\theta} E(r_n)|$ decays. Note that we do not need to differentiate between the sets $\pi_{\theta} (E(r))$ and $(\pi_{\theta} E)(r)$ (that is, the neighborhood of a projection within a line); they are equal.

First, note that $E$ has Hausdorff dimension at most $1$; otherwise, a result of Marstrand in \cite{Mar54}, Chapter II, would imply that $\fav(E) > 0$, contradicting that $\fav(E) \le \fav(E(r_n)) \le r_n^{1 - s} \to 0$. (Of course, this follows from Mattila's work in \cite{Mat90} or Chapter 9 of \cite{Mat95}; however, we are trying to avoid the use of potentials). For each $n$, we can consider a set of angles
$$A_n = \left\{\theta \in A : \dimh(\pi_{\theta} E) = \dimh E \text{ and } |\pi_{\theta} E(r_n)| \le r_n^{s - \epsilon}\right\}.$$
The first condition holds for almost all $\theta$; this also follows from Chapter II of \cite{Mar54}. Secondly, since $\int_A |\pi_{\theta} E(r_n)| \le C r_n^s$, we can estimate the size of the exceptional set $A_n^c$ by
$$|A \cap A_n^c| \le C r_n^{\epsilon}.$$
Passing to a subsequence of scales (which we also denote as $r_n$) if necessary, we can assume that $\sum_n |A \setminus A_n| < |A|$; thus, there exists an angle $\varphi \in \bigcap_n A_n$. In particular, $\pi_{\varphi} E$ has Hausdorff dimension equal to that of $E$ itself.

Next, we will control the Hausdorff measure of $\pi_{\varphi} E$ at dimensions a little above $s$. Note that $\pi_{\varphi} E(r_n)$ consists of a union of disjoint intervals $I_{n, k}$, each having length at least $2 r_n$. This forms a natural cover of $\pi_{\varphi}(E)$. We can estimate the number of intervals in the cover via
$$r_n^{s - \epsilon} \ge |\pi_{\varphi} E(r_n)| = \sum_k |I_{n, k}|$$
Using that $|I_{n, k}| \gtrsim r_n$, we can rearrange this to find that there are at most $r_n^{s - \epsilon - 1}$ such intervals. We are now in a position to estimate sums of the form $\sum_k |I_{n, k}|^p$ for $p \in (0, 1)$. A direct application of H\"older's inequality shows that if $p \in (0, 1)$, $q$ satisfies $1/p - 1/q = 1$, and $\mu$ is any measure,
$$\int fg \, d\mu \ge \left(\int f^p \, d\mu\right)^{1/p} \left(\int g^{-q} \, d\mu \right)^{-1/q}$$	
holds for measurable functions. Taking this with counting measure, we find that
\begin{align*}
r_n^{s - \epsilon} &\ge \sum_k |I_{n, k}| \\
&\ge \left(\sum_k |I_{n, k}|^p\right)^{1/p} \left(\sum_k 1^{-q}\right)^{-1/q} \\
&\gtrsim \left(\sum_k |I_{n, k}|^p\right)^{1/p} \left(r_n^{s - \epsilon - 1}\right)^{-1/q}
\end{align*}
Rearranging this leads to
$$\left(\sum_k |I_{n, k}|^p\right)^{1/p} \lesssim r_n^{s - \epsilon + \frac 1 q(s - \epsilon - 1)}.$$
The exponent can be simplified as
$$\left(1 + \frac 1 q\right)(s - \epsilon) - \frac 1 q = \frac{1}{p} \left(s - \epsilon - (1 - p)\right).$$
As long as $s - \epsilon - (1 - p) \ge 0$, we can give a uniform upper bound on $\sum_k |I_{n, k}|^p$; this works provided that $p \ge 1 - s + \epsilon$. Furthermore, one can see that each $I_{n, k}$ has radius no larger than $r_n^{s - \epsilon}$, which tends to zero as $n$ grows. Combining these observations leads to
$$\mathcal{H}^{1 - s + \epsilon}\left(\pi_{\varphi} E\right) < \infty.$$
so that $\dimh \pi_{\varphi}(E) \le 1 - s + \epsilon.$

Finally, take a smaller $\epsilon$ and rerun the argument with a (potentially) new choice of $\varphi$. Taking a sequence $\epsilon_m \to 0$, we then get a sequence $\varphi_m$ of angles and
$$\dimh E = \dimh \pi_{\varphi}(E) \le 1 - s + \epsilon_m \to 1 - s.$$
This is the desired bound on dimension.
\end{proof}

\section{Self-similar Sets}
In this section, we will show how self-similarity can be used to give lower bounds on the Favard length, as well as control the behavior of the sequence of projection lengths. Our result is

\begin{theorem}
Suppose that $\{A_n\}_{n \in \mathbb{N}}$ is a sequence of sets such that $A_{n + 1} \subseteq A_{n}$ for all $n$, that each generation can be written as a union
$$A_{n + 1} = \bigcup_{i = 1}^N r_{i} A_n + \beta_{i}$$
for some fixed set of contraction ratios $r_i > 0$, and that $\sum_{i} r_{i} = 1$. Then for each $\theta$, the sequence $\{|\pi_{\theta} A_n|\}_{n \in \mathbb{N}}$ is convex.
\end{theorem}

\begin{proof} Let us define $E_{n, \theta} = \pi_{\theta} A_n$; note that on the projection side, $E_{n, \theta}$ is also self-similar and is generated by similitudes of the form $T_i : x \mapsto r_i x + \pi_{\theta} \beta_i$, where $x$ is measured within the line $L_{\theta}$. We then have
\begin{align*}
\alpha_n(\theta) - \alpha_{n + 1}(\theta) &= |E_{n, \theta}| - |E_{n + 1, \theta}| \\
&= |E_{n, \theta} \setminus E_{n + 1, \theta}| \\
&= \left|\bigcup_{i = 1}^N T_i(E_{n - 1, \theta}) \setminus \bigcup_{j = 1}^N T_j(E_{n, \theta})\right| \\
&\le \left|\bigcup_{i = 1}^N \big(T_i(E_{n - 1, \theta}) \setminus T_i(E_{n, \theta})\big) \right| \\
&= \left|\bigcup_{i = 1}^N T_i(E_{n - 1, \theta} \setminus E_{n, \theta})\right| \\
&\le \sum_{i = 1}^N r_i |E_{n - 1, \theta} \setminus E_{n, \theta}| \\
&= \alpha_{n - 1}(\theta) - \alpha_n(\theta).
\end{align*}
where we have used that $\{E_n\}_{n \in \mathbb{N}}$ is a decreasing sequence of sets, that each $T_i$ is a contraction by $r_i$ along with a translation, and that $\sum_i r_i = 1$. Rearranging this, we find that $$\alpha_n(\theta) \le \frac{\alpha_{n - 1}(\theta) + \alpha_{n + 1}(\theta)}{2}$$
which is the desired result.
\end{proof}

As a corollary, we can easily deduce lower bounds on the Favard length of the four-corner Cantor set. Recall that the generations of this set are constructed by taking $\mathcal{K}_0 = [0, 1]^2$; then $\mathcal{K}_n$ is constructed by taking each square in $\mathcal{K}_{n - 1}$, dividing it into four sections horizontally and vertically, and taking the four subsquares at the corners.  The result is a family of $4^n$ squares of sidelength $4^{-n}$ each. Alternatively, the set is generated by the similitudes $f_i(z) = \frac 1 4 z + \beta_i$, with $\{\beta_i : 1 \le i \le 4\} = \{(0, 0), (0, 3/4), (3/4, 0), (3/4, 3/4)\}$. We have the following result:

\begin{corl}
The Favard lengths of the generations of the four-corner Cantor set satisfy $\fav(\kn) \gtrsim \frac 1 n$.
\end{corl}

\begin{proof}
Fix $n \in \mathbb{N}$. Note that if we take $\theta^* = \arctan 1 / 2$, the projection $\pi_{\theta^*}$ maps the four components of $\mathcal{K}_1$ to four intervals that only overlap on the boundaries (and therefore, $\pi_{\theta} \mathcal{K}_n$ is the same interval for all $n$, as an application of self-similarity). Therefore, $\alpha_0(\theta^*) - \alpha_1(\theta^*) = 0$. Furthermore, $\theta \mapsto \alpha_0(\theta) - \alpha_1(\theta)$ is piecewise $C^1$ and the derivative is bounded by $10$ (which follows from a direct computation of the function $\alpha_0 - \alpha_1$). Hence, there is an interval $I_n$ of length $\frac{1}{20n}$ centered at $\theta^*$ such that 
$$0 \le \alpha_0(\theta) - \alpha_1(\theta) \le \frac{1}{2n}$$
for all $\theta$ in the interval. Applying convexity iteratively leads to $\alpha_n(\theta) \ge \frac 1 2$ for all $\theta \in I_n$, and so
$$\fav(\kn) \ge \int_{I_n} \alpha_n(\theta) \, d\theta \ge \frac 1 {40n}$$
as desired.
\end{proof}

Note that the key idea here is that there is a special angle at which the projection acts (more or less) bijectively on components. It follows that this technique is applicable to a broad class of self-similar sets with such an angle - the Sierpinski gasket is another important example. One hopes that tightening the losses of this technique would be sufficient to improve the estimate past $1/n$; it was proved in \cite{BatVol08} that the Favard length of $\kn$ is actually at least $c \ln n / n$.

It is worth mentioning that the function $n \mapsto |E(4^{-n})|$ is not generally convex without the self-similarity assumption. For example, the set $$\{0, 1/4, 1/2, 3/4, ..., 100\}$$ (or a small neighborhood of it) serves as a counterexample. A modification of this example (using ever finer lattices around carefully selected points in the set), one can find examples where $n \mapsto |E(4^{-n})|$ is neither eventually convex nor eventually concave. Rather, the sequence exhibits ``see-saw'' behavior as it decays to zero.

\bibliography{cite}{}
\bibliographystyle{plain}
\end{document}